\documentclass[11pt]{amsart}
\usepackage{amssymb}
\usepackage{amsmath}
\usepackage{amsfonts}
\usepackage{graphicx}

\usepackage[total={17cm,22cm},top=2.5cm, left=2.3cm]{geometry}
\usepackage{hyperref}
    \usepackage{aeguill}
    \usepackage{type1cm}
\usepackage{enumitem}
\usepackage{crossreftools}

\theoremstyle{plain}
\newtheorem{thm}{Theorem}[section]

\newtheorem{lemma}[thm]{Lemma}

\newtheorem{remark}[thm]{Remark}

\newtheorem{theorem}[thm]{Theorem}

\numberwithin{equation}{section}
\newcommand{\N}{\mathbb{N}}

\newcommand{\R}{\mathbb{R}}

\newcommand{\Rn}{\mathbb{R}^n}

\DeclareMathOperator{\lip}{Lip\,\!}

\DeclareMathOperator{\sop}{supp\,\!}

\usepackage[usenames,dvipsnames]{color}

\usepackage[dvipsnames]{xcolor}

\newcommand{\ud}[0]{\,\mathrm{d}}

\begin{document}
\title[The sharp Whitney Extension theorem for convex $C^1$ Lipschitz functions]{The sharp Whitney Extension theorem for convex $C^1$ Lipschitz functions}

\author[Carlos Mudarra]{Carlos Mudarra}
\address{Department of Mathematical Sciences, Norwegian University of Science and Technology, 7941 Trondheim, Norway}
\email{carlos.mudarra@ntnu.no}

\date{\today}

\makeatletter
\@namedef{subjclassname@2020}{{\mdseries 2020} Mathematics Subject Classification}
\makeatother

\keywords{convex function, Lipschitz function, sharp constant, Whitney extension theorem}

\subjclass[2020]{26A16, 26B05, 26B25, 52A41, 54C20, 90C25}


\begin{abstract}
For an arbitrary set $E \subset \mathbb{R}^n$, and functions $f:E \to \mathbb{R}$, $G: E\to \mathbb{R}^n$ with $G$ bounded, we construct $C^1(\mathbb{R}^n)$ convex extensions $(F, \nabla F)$ of $(f,G)$ with the sharp Lipschitz constant
 $$
 \mathrm{Lip}(F) = \sup_{x\in E} |G(x)|,
 $$
 provided that $(f,G)$ satisfies the pertinent necessary and sufficient conditions for $C^1$ convex, and Lipschitz extendability. Also, these extensions can be constructed with prescribed global behavior in terms of directions of coercivity.  
\end{abstract}

\maketitle

\section{Introduction and main results}
Given a set $A \subset \R^n$ and a function $g: A \to \R,$ we will denote the Lipschitz constant of $g$ on $A$ by
$$
\lip(g,A) : = \inf\lbrace C \geq 0 \, : \, |g(x)-g(y)| \leq C |x-y| \: \text{ for all } \: x,y \in A \rbrace .
$$
We will abbreviate by $\lip(g) =\lip(g,\R^n) $ in the case where $A= \R^n.$ If $L \geq 0,$ we will say that $g$ is $L$-Lipschitz on $A$ when $\lip(g,A) \leq L. $ Also, by a $1$-jet on a subset $E$ of $\R^n$ we understand a couple $(f,G)$ of functions $f:E \to \R$, $G: E \to \R^n.$

The classical Whitney Extension Theorem \cite{W34} for $C^1$ provides necessary and sufficient conditions on a jet $(f,G)$ for the existence of $C^1(\R^n)$ functions $F$ with $(F,\nabla F) = (f,G)$ on $E.$ We then say that $F$ (or $(F,\nabla F)$) is an extension of the jet $(f,G).$

\medskip

In this paper, we prove a Whitney Extension Theorem for \textit{convex} $C^1(\R^n)$ Lipschitz functions, where the extensions have the best possible Lipschitz constants $\lip(F) = \|G\|_{L^\infty(E)}.$

In last decades, plenty of Whitney-type problems have been thoroughly studied and solved for various classes of jets and functions. This includes results for $C^m$ or $C^{m,\omega}$ functions \cite{ALeGM18,AM21,F05,F06,F07,FK09,FS18,JSSG13,JZ25,LeG09,W73}, convex functions \cite{ALeGM18,AM17,AM19APDE,AM19,DHLL18,DM25,M24}, Sobolev functions \cite{FIL14,I13,S17}, or smooth functions in the Heisenberg group \cite{PSZ24,PSZ19,Z23}. This is by no means an exhaustive list.

\smallskip

In \cite{AM17}, we obtained a convex $C^1$ extension theorem for $1$-jets defined on compact subsets $E$ of $\R^n.$ The extendability criteria for a jet $(f,G)$ consists of a necessary and sufficient couple of conditions, namely,
\begin{equation}\label{eq:condition_C_compact}
    f(x) \geq f(y) + \langle G(y), x-y \rangle \quad \text{for all } \:  x,y\in E, \tag{$C$}
\end{equation}
\begin{equation}\label{eq:condition_CW1}
 x,y \in E, \quad    f(x)=f(y) + \langle G(y), x-y \rangle \implies G(x)=G(y). \tag{$CW^1$}
\end{equation}
Then the extensions $F$ of a jet $(f,G)$ could be taken to satisfy $\lip(F) \leq c(n) \sup_{x\in E} |G(x)|$, for a dimensional constant $c(n)$ with $c(n) \to \infty$ as $n \to \infty.$ By different means, it was later shown in \cite{AM20} that the dimensional factor $c(n)$ can be replaced by $c=5.$ The same problem in the case where $E$ is an arbitrary (possibly unbounded) subset of $\R^n$ was solved in \cite{AM19APDE}, with bounds of the form $\lip(F) \leq c(n) \sup_{x\in E} |G(x)|$ in the case of Lipschitz functions. Unlike the cases of compact sets $E$ or the $C^{1,\omega}$ regularity classes, this problem was much more difficult to solve, especially because of the presence of the so-called \textit{corners at infinity} and their possible incompatibility with everywhere differentiable convex extensions. We also refer to \cite{AM19} for $C^\infty$ convex extensions of jets defined in compact convex subsets of $\R^n$. 

\smallskip

In a very recent work \cite{DM25}, we managed to construct $C^1$ and $C^{1,\omega}$ convex extensions with the best possible Lipschitz constants, even in the setting of super-reflexive Banach spaces. More precisely, if $X$ is super-reflexive, and $E \subset X$ is compact, any jet $(f,G)$ defined on $E$ admits a convex extension $(F, DF)$ of class $C^1(X)$ so that $\lip(F,X) = \sup_{y\in E} \|G(y)\|_*$, provided $(f,G)$ satisfies the necessary conditions \eqref{eq:condition_C_compact} and \eqref{eq:condition_CW1}. This result corresponds to \cite[Theorem 1.6]{DM25} and its proof is based on an application of a similar theorem for $C^{1,\omega}$ convex extensions, for appropriate moduli of continuity $\omega$; see \cite[Theorem 1.2]{DM25}.

\smallskip

We learnt from \cite{AM19APDE} that extending jets $(f,G)$ with $C^1$ convex functions from unbounded sets $E$ of $\R^n$ is a much more complicated task than the bounded case. Conditions \eqref{eq:condition_C_compact} and \eqref{eq:condition_CW1} are no longer sufficient when $E$ is an unbounded set, and we need a careful study of the possible \textit{corners at infinity} of the jets and the \textit{global shape of convex functions} in terms of directions of coercivity. 

\smallskip

Given a linear subspace $X$ of $\R^n$, we will denote by $X^\perp$ its orthogonal complement, and by $P_X : \R^n \to X$ the orthogonal projection onto $X.$ In the terminology of \cite{AM19APDE}, we say that a jet $(f,G) : E \to \R \times \R^n$ has a \textit{corner at infinity in a direction of $X^\perp$} when there are two sequences $(x_j)_j,$ $(z_j)_j$ in $E$ with $(|x_j|)_j \to \infty$, both $(G(z_j))_j$ and $(P_X(x_j))_j$ bounded, and 
$$
\lim_{j \to \infty} \left( f(x_j)-f(z_j)- \langle G(z_j), x_j-z_j \rangle \right)=0  \quad  \text{and} \quad \limsup_{j \to \infty} |G(x_j)-G(z_j)|>0.
$$
Also, according to \cite[Lemma 4.1]{A12} or \cite[Theorem 1.11]{AM19APDE}, for every convex function $F$ (not necessarily of class $C^1(\R^n)$) there is a unique linear subspace $X_F$ of $\R^n$, a convex coercive function $c: X_F \to \R$, and $v\in \R^n$ so that
\begin{equation}\label{eq:decomposition_intro}
F(x) = c(P_{X_F}(x)) + \langle v, x \rangle \quad \text{for all} \quad x\in \R^n.    
\end{equation}
Coercivity simply means that $\lim_{|x| \to \infty} c(x) = + \infty.$ It was implicitly shown in \cite{AM19APDE} that in fact 
\begin{equation}\label{eq:formula_X_F_intro}
 X_F = \mathrm{span}\lbrace \xi_x-\xi_y \, : \, \xi_x \in \partial F(x) , \, \xi_y\in \partial F(y), \, x,y\in \R^n \rbrace.   
\end{equation} 
See for example Lemma \ref{lem:properties_subdifferentials_decompositions} below for a detailed proof. Here $\partial F(x)$ denotes the subdifferential of $F$ at $x\in \R^n.$ The subspace $X_F$ determines the directions of coercivity of $F$ in the sense of \eqref{eq:decomposition_intro}. For differentiable functions, the connection with corners at infinity as follows: if $F \in C^1(\R^n)$, then the jet $(F, \nabla F)$ cannot have a corner at infinity in a direction of $(X_F)^\perp. $
Moreover, if $\nabla F=G$ on $E,$ by formula \eqref{eq:formula_X_F_intro} we have that $X_F$ must contain the subspace $Y(E,f,G):=\mathrm{span}\lbrace G(x)-G(y) \, : \, x,y\in E \rbrace.$ Therefore, the existence of a linear subspace $X$ of $\R^n$ that contains $Y$, and so that $(f,G)$ has no corners at infinity in directions of $X^\perp$ is a necessary condition for convex $C^1(\R^n)$ extensions of $(f,G).$
However, unless $X=Y(E,f,G),$ further conditions are needed to guarantee the existence of such extensions, as the shape of $(f,G)$ on $E$ might force all convex extensions of $(f,G)$ to be non-differentiable along a half-line of $\R^n.$ We remedied this issue by introducing necessary conditions that, roughly speaking, led us to an extension of the data $(E,f,G)$ with finitely many more points so that the new jet $(E^*,f^*,G^*)$ does satisfy $Y(E^*,f^*,G^*)=X.$  

\smallskip

Interestingly, those extensions can be taken with a \textit{prescribed global behavior} in terms of directions of coercivity, again, in the sense of \eqref{eq:decomposition_intro}. More precisely, the conditions characterized jets $(f,G)$ along with subspaces $X \subset \R^n$ for the existence of convex $F \in C^1(\R^n)$ with $(F,\nabla F)=(f,G)$ on $E$ and $X_F=X.$ This corresponds to \cite[Theorem 1.13]{AM19APDE} for all convex functions, and \cite[Theorem 1.14]{AM19APDE} for Lipschitz convex functions. These results where the extensions are built with customized directions of coercivity have already been relevant in obtaining Lusin-type theorems for convex functions; see \cite{ADH24,AH21}.

\smallskip

In this paper, we construct convex $C^1(\R^n)$ Lipschitz extensions of $1$-jets $(f,G)$ with the best possible Lipschitz constant $\lip(F) = \sup_{x\in E} |G(x)|$, which additional prescribe a given global behavior, thus obtaining the sharp form of \cite[Theorem 1.14]{AM19APDE}. The extendability conditions we will be using are the same as in that theorem, namely,




\begin{enumerate}[align=left]
\item[{\crtcrossreflabel{$(\mathrm{C_b})$}[condition:convexity]}] $G$ is continuous and bounded and $f(x)\geq f(y)+\langle G(y), x-y\rangle$ for all $x, y\in E$.

\medskip

\item[{\crtcrossreflabel{$(\mathrm{S}_{Y,X})$}[condition:subspacecontained]}] $Y:=\mathrm{span} \{G(x)-G(y) : x, y\in E\} \subseteq X.$   

\medskip

\item[{\crtcrossreflabel{$(\mathrm{C}_{Y,X})$}[condition:existencecones]}] If $Y$ is as in \ref{condition:subspacecontained} and $Y\neq X$, and we denote $k= \dim Y$ and $d= \dim X$, there are points $p_1, \ldots, p_{d-k} \in \R^n \setminus \overline{E}$, a number $\varepsilon \in (0,1),$ and linearly independent normalized vectors $w_1, \ldots, w_{d-k} \in X \cap Y^{\perp}$ such that, for every $j=1, \ldots , d-k,$ the cone
$$
 V_j := \lbrace x\in \R^n \: : \:  \varepsilon \langle w_j,x- p_j \rangle \geq |P_Y(x-p_j)|\rbrace
$$ 
satisfies $V_j\cap \overline{E}= \emptyset.$ 

\medskip

\item[{\crtcrossreflabel{$(\mathrm{CW}^1_X)$}[condition:cornersorthogonal]}] If $(x_j)_j, (z_j)_j$ are sequences in $E$ such that $(P_X(x_j))_j$ is bounded, then
$$
\lim_{j\to\infty} \left( f(x_j)-f(z_j)- \langle G(z_j), x_j-z_j \rangle \right) = 0 \implies \lim_{j\to\infty} | G(x_j)-G(z_j) | = 0.
$$
\end{enumerate}

Condition \ref{condition:cornersorthogonal} says that the jet $(f,G)$ cannot have corners at infinity in directions orthogonal to $X$. The most technical condition \ref{condition:existencecones} essentially says that if  $Y \neq X,$ then there should be enough room in $\R^n \setminus \overline{E}$ to insert $d-k$ cones \textit{orthogonal to} $Y$.

\medskip

The main result of this paper is as follows.

\begin{theorem}\label{thm:maintheorem}
Let $E \subset \Rn$ be a subset, $X\subset\R^n$ a linear subspace, and $f:E \to \R, \: G: E \to \Rn$ a $1$-jet with $G$ non-constant. Assume that $f,G,X$ satisfy conditions \ref{condition:convexity}, \ref{condition:subspacecontained}, \ref{condition:existencecones}, \ref{condition:cornersorthogonal} in the set $E$. Then there exists a Lipschitz convex function $F:\R^n\to\R$ of class $C^1(\R^n)$ such that $(F, \nabla F)=(f,G)$ on $E$, $ X_F=X,$ and
\begin{equation}\label{eq:sharp_constant_theorem}
 \lip(F) = \sup_{x\in \R^n} | \nabla F(x)| = \sup_{y\in E} |G(y)|.   
\end{equation}

\end{theorem}

We already know from \cite{AM19APDE} that \ref{condition:convexity}, \ref{condition:subspacecontained}, \ref{condition:existencecones}, \ref{condition:cornersorthogonal} are necessary and sufficient conditions for the existence of a convex, Lipschitz and $C^1$ function $F: \R^n \to \R$ with $(F, \nabla F) =(f,G)$ on $E$ and $X_F=X.$ However, in order to construct an extension with the sharp Lipschitz condition \eqref{eq:sharp_constant_theorem}, we will need to: 1) establish a decomposition theorem as in \eqref{eq:decomposition_intro} for Lipschitz convex functions with a good estimate for $|v|$; 2) modify the construction of the (non-sharp) extension in \cite[Theorem 1.14]{AM19APDE} at several points; 3) apply certain Lipschitz envelope to that non-sharp extension. See the outline of the proof of Theorem \ref{thm:maintheorem} below.

\smallskip

Let us make some remarks concerning the identity \eqref{eq:sharp_constant_theorem} in Theorem \ref{thm:maintheorem}. 

\begin{remark}
    {\em Let $E \subset \R^n$ and $(f,G) : E \to \R \times \R^n$ a $1$-jet. 

\begin{enumerate}
    \item For non-compact sets $E,$ the best estimate for the Lipschitz constant $\lip(F)$ of a convex $C^1(\R^n)$ extension of $(f,G)$ before Theorem \ref{thm:maintheorem} were of the form $\lip(F) \leq c(n) \sup_{y\in E} |G(y)|$, with $c(n)$ growing to infinity with the dimension; cf. \cite[Theorem 1.14]{AM19APDE}. And for compact sets $E,$ the optimal bound has been recently obtained in \cite[Theorem 1.6]{DM25}, in the setting of super-reflexive Banach spaces.
    \item The trivial case where $G$ is constant is excluded from Theorem \ref{thm:maintheorem}, because if $G(y) = u \in \R^n$ for all $y\in E,$ condition \ref{condition:convexity} easily yields
    $
     f(x) - f(y) = \langle u,x -y \rangle
    $
    for all $x,y\in E.$ Therefore, $(f, G)$ already admits a trivial affine extension satisfying \eqref{eq:sharp_constant_theorem}, namely, the function $F(x) = f(y_0)+ \langle u, x-y_0 \rangle$ with any $y_0 \in E.$ However, in this trivial case it is not always possible to find extensions $F$ satisfying \eqref{eq:sharp_constant_theorem} and at the same time prescribing a global behavior $X \neq \lbrace 0 \rbrace$; e.g. in the case where $G$ is identically zero on $E.$
    \item Condition \ref{condition:convexity} ensures that $\sup_{y\in E} |G(y)|$ is the right constant to gauge the Lipschitz properties of the jet $(f,G).$ Indeed, it is easy to see that
    $$
    |f(x)-f(z)| \leq \max\lbrace \langle G(x), x-z \rangle , \langle G(z), z-x \rangle \rbrace \leq \big ( \sup_{y\in E} |G(y)| \big ) |x-z|, \quad x,z\in E,
    $$
    whence $\lip(f,E) \leq \sup_{y\in E} |G(y)| .$
    \item For general (not necessarily convex) $C^1$ functions, the known sharpest form of the Whitney-type extensions is that from \cite{JSSG13}, which provide, for every $\varepsilon>0,$ an extension $F_\varepsilon \in C^1(\R^n)$ of $(f,G)$ with
    $$
    \lip(F_\varepsilon,\R^n) \leq \varepsilon + \max \left\lbrace \lip(f,E), \sup_{y\in E} |G(y)| \right \rbrace .
    $$
    This result holds even if one replaces $\R^n$ with a Hilbert space, under appropriate necessary and sufficient conditions. Remarkably, the identity \eqref{eq:sharp_constant_theorem} in Theorem \ref{thm:maintheorem} shows that that $\varepsilon>0$ can be \textit{done away with} for the analogous convex extension problem. If we disregarded prescription of global behavior and were interested in extensions that are optimal \textit{only} up to some additive $\varepsilon>0,$ the proof of Theorem \ref{thm:maintheorem} would be considerably shorter; see Remark \ref{rem:1+epsilon}.   
\end{enumerate}
    
    }
\end{remark}

\subsection*{Outline of the proof of Theorem \ref{thm:maintheorem}} We will first show in Section \ref{sect:behavior} that non-affine $L$-Lipschitz convex functions can be written as $f= c\circ P_X  + \langle v, \cdot \rangle,$ with $c:X \to \R$ convex and coercive and $|v|<L;$ see Theorem \ref{thm:decomposition_controlleddirection}. This is achieved by first proving the result for maxima of affine functions, where we need to use the strict convexity of the Euclidean norm in $\R^n \setminus \lbrace 0 \rbrace.$ The proof of Theorem \ref{thm:maintheorem} is given in Section \ref{sect:proofmain}, and split into several subsections. For a subspace $X \subset \R^n$ and a jet $(f,G)$ with $|G|$ bounded by $L$ on $E$, we first define a \textit{minimal convex extension} (not necessarily everywhere differentiable) from the jet $(f,G)$ by formula 
$$
m(x) = \sup\lbrace f(y)+ \langle G(y),x-y \rangle \, : \, y \in E \rbrace, \quad x\in \R^n,
$$
and decompose it according to Theorem \ref{thm:decomposition_controlleddirection}. That is, $m = c \circ P_Y + \langle v, \cdot \rangle;$ where $Y$ is that of condition \ref{condition:subspacecontained} and $|v|<L$. 
Then, using this strict inequality and condition \ref{condition:existencecones}, we extend $(f,G)$ from $E$ to $E^* = E \cup \lbrace q_1, \ldots, q_{d-k} \rbrace$, where $q_j \in V_j,$ and so that the expanded jet $(f^*,G^*)$ satisfies both $X= \mathrm{span}\lbrace G^*(x)-G^*(y) \, : \, x,y\in E^* \rbrace$ and condition \ref{condition:cornersorthogonal}, \textbf{and still $|G^*| \leq L$ on $E^*.$}
This will imply that the updated minimal convex extension
$$
m^*(x) = \sup\lbrace f^*(y)+ \langle G^*(y),x-y \rangle \, : \, y \in E^* \rbrace, \quad x\in \R^n,
$$
is $L$-Lipschitz, and so can be written as $m^* = c \circ P_X + \langle v^*, \cdot \rangle,$ with $|v^*| < L,$ and $c^* : X \to \R$ convex, coercive and differentiable in $\overline{P_X(E^*)}.$ Then we will find a convex $C^1(X)$ and coercive function $H : X \to \R$ with $(H, \nabla H ) = (c,\nabla c)$ on $P_X(E^*).$ Defining 
$$
\widetilde{F}(x) = H(P_X(x)) + \langle v^*, x \rangle, \quad x\in \R^n,
$$
we obtain a convex $C^1(\R^n)$ function with $(F,\nabla F) = (f^*,G^*) $ on $E^*$ and $|v^*| \leq L$. 
The desired function in Theorem \ref{thm:maintheorem} will be
$$
F(x)= \inf\lbrace \widetilde{F}(y) + L |x-y| \, : \, y \in \R^n \rbrace, \quad x\in \R^n.
$$
In order to show that $X_F=X$, it is crucial that $m^*$ is precisely $L$-Lipschitz, which in turn implies $m^* \leq F$. We then verify the rest of the properties of $F$, the most difficult one being the proof of the differentiability of $F$ at every $x\in \R^n.$
The first key idea here is that formula $\widetilde{F} = H \circ P_X + \langle v^*, \cdot \rangle$ and the fact that $H$ is coercive and $|v^*| \leq L$ allows to approximate the infimum $F(x)$ with sequences $(y_k)_k$ so that $(P_X(y_k))_k$ is bounded.
On the other hand, $X_{\widetilde{F}}=X$ and so $\widetilde{F}$ cannot have corners at infinity in directions orthogonal to $X.$ These two facts are (implicitly) combined to appropriately estimate certain finite differences that appear in the proof of the differentiability.

\section{Global behavior of Lipschitz convex functions}\label{sect:behavior}

In this section we show that non-affine, convex and Lipschitz functions $f$ can be written as $f = c \circ P_X + \langle v, \cdot \rangle$, where $c: X \to \R$ is convex and coercive, and $|v|< \lip(f).$ This result can be seen as a Lipschitz version of \cite[Theorem 1.11]{AM19APDE} or \cite[Lemma 4.2]{A12}.

\medskip

We first show that such a decomposition exists for convex functions that are a finite maximum of affine functions, also called \textit{corner functions} in the terminology of \cite{AM19APDE}.

\begin{lemma}\label{lem:maxfunction_esscoercive}
Let $1 \leq  k \leq n$, vectors $v_1, \ldots, v_{k+1} \in \R^n$ and $a_1,\ldots,a_{k+1} \in \R$ such that $ \lbrace v_j-v_1 \rbrace_{j=2}^{k+1}$ are linearly independent in $\R^n.$ Then the function
\begin{equation}\label{eq:def_maxfunction_C}
C(x) = \max \lbrace \langle v_j,x \rangle +a_j \: : \: j=1, \ldots, k+1 \rbrace, \quad x\in \Rn,    
\end{equation}
can be written in the form
$$
C(x) = c(P_X(x)) + \langle v, x \rangle, \quad x\in \R^n,
$$
where
\begin{equation} \label{eq:formulas_subspace_vector_maxfunctions}
  X= \mathrm{span}\lbrace v_j-v_1 \: : \: j=2,\ldots,k+1 \rbrace  \quad \text{and} \quad v = \frac{1}{k+1} \sum_{j=1}^{k+1} v_j,  
\end{equation}
and $c : X \to \R$ is convex and coercive in $X$. Moreover, if $L:= \max \lbrace |v_j| \, : \, j=1, \ldots, k+1 \rbrace$, one has
$$
\lip(c,X) \leq L + |v| \quad \text{and} \quad |v|<L.
$$
\end{lemma}
\begin{proof}
Let $C: \R^n \to \R$, $X \subset \R^n$ and $v\in \R^n$ as in \eqref{eq:def_maxfunction_C} and \eqref{eq:formulas_subspace_vector_maxfunctions}. Define the vectors $u_j:=v_1-v_j$ for $j=2,\ldots, k+1,$ and note that
$$
v = \frac{1}{k+1} \sum_{i=1}^{k+1} v_i = \frac{1}{k+1} \sum_{i=1}^{k+1}(v_1-u_i) = v_1-\frac{1}{k+1} \sum_{i=1}^{k+1} u_i.
$$
Using this identity, we can rewrite $C(x),$ $x\in \R^n,$ as follows
\begin{align*}
C(x)-\langle v,x \rangle & = \max \left\lbrace \langle v_j-v_1,x \rangle + b_j -\frac{1}{k+1} \sum_{i=1}^{k+1} \langle u_i, x \rangle  \: : \: j=1,\ldots, k+1 \right\rbrace \\
& = \max \left\lbrace \frac{1}{k+1}  \sum_{i=2}^{k+1} \langle u_i,x \rangle + b_1,-\langle u_j,x \rangle + \frac{1}{k+1} \sum_{i=2}^{k+1} \langle u_i,x \rangle + b_j \: : \: j=2,\ldots, k+1 \right\rbrace .
\end{align*}
Because $u_j \in X,$ we have that $\langle u_j, x \rangle = \langle u_j, P_X(x) \rangle$, which allows to decompose $C$ in the form
\begin{equation}\label{eq:decomposition_maxfunction_cPX}
 C(x)= c (P_X(x))+ \langle v,x \rangle , \quad x\in \R^n;     
\end{equation}
where
\begin{equation}\label{eq:coercivepart_maxfunction}
c(y)= \max \left\lbrace   \langle \xi,y \rangle + b_1, \langle \xi-u_j,y \rangle + b_j \: : \: j=2,\ldots, k+1 \right\rbrace, \quad y \in X, \quad \xi:=\frac{1}{k+1} \sum_{i=2}^{k+1} \langle u_i,y \rangle.
\end{equation}
As a maximum of affine functions, $c:X \to \R$ is convex in $X$. We next show that $c$ is a coercive function in $X$. To do so, we may of course assume that $b_1=\cdots=b_{k+1}=0,$ so that $c$ is simply 
$$
c(y) = \max \lbrace \langle \xi, y \rangle, \langle \xi-u_j,y \rangle + b_j \: : \: j=2,\ldots, k+1 \rbrace, \quad y \in X.
$$
The family $\lbrace u_j\rbrace_{j=2}^{k+1}$ is a basis of $X,$ and so
$$
\| y \|: = \max \lbrace | \langle u_2, y \rangle |, \ldots, |\langle u_{k+1}, y \rangle | \rbrace, \quad y \in X,
$$
defines an equivalent norm in $X.$ Given $y\in X$, assume first that $\langle \xi, y \rangle \geq \langle \xi-u_j, y\rangle$ for all $j=2,\ldots, k+1.$ We have that $\langle u_j, y \rangle \geq 0$ for all $j$, and therefore
$$
c(y) = \langle \xi, y \rangle  \geq \frac{1}{k+1} \max \lbrace | \langle u_2, y \rangle |, \ldots, |\langle u_{k+1}, y \rangle | \rbrace = \frac{1}{k+1}\|y\|.
$$
And in the case where the maximum defining $c(y)$ is $\langle \xi-u_j, y \rangle $ for some $j \in \lbrace 2, \ldots , k+1 \rbrace$, then we have the inequality
$$
\langle u_j , y \rangle \leq \min \lbrace 0 , \langle u_l, y \rangle \: : \: l=2, \ldots, k+1 \rbrace.
$$
This gives the estimates
\begin{align*}
    c(y)  = \langle \xi-u_j, y \rangle & = - \langle u_j, y \rangle + \frac{1}{k+1} \sum_{ \lbrace l \: : \: \langle u_l,y \rangle \leq 0 \rbrace }\langle u_l, y \rangle+ \frac{1}{k+1} \sum_{ \lbrace l \: : \: \langle u_l,y \rangle > 0 \rbrace }\langle u_l, y \rangle    \\
    & \geq -\frac{1}{k} \sum_{ \lbrace l \: : \: \langle u_l,y \rangle \leq 0 \rbrace }\langle u_l, y \rangle + \frac{1}{k+1} \sum_{ \lbrace l \: : \: \langle u_l,y \rangle \leq 0 \rbrace }\langle u_l, y \rangle+ \frac{1}{k+1} \sum_{ \lbrace l \: : \: \langle u_l,y \rangle > 0 \rbrace }\langle u_l, y \rangle \\
    & = \frac{1}{k(k+1)} \sum_{ \lbrace l \: : \: \langle u_l,y \rangle \leq 0 \rbrace }|\langle u_l, y \rangle | + \frac{1}{k+1} \sum_{ \lbrace l \: : \: \langle u_l,y \rangle > 0 \rbrace }|\langle u_l, y \rangle| \\
    & \geq \frac{1}{k(k+1)} \max \lbrace | \langle u_2, y \rangle |, \ldots, |\langle u_{k+1}, y \rangle | \rbrace = \frac{1}{k(k+1)} \|y\|.
\end{align*}
These inequalities yield the coercivity of $c$ in $X.$

\smallskip

Let us show the inequality $|v|<L$. Since the family $\lbrace v_j-v_1 \rbrace_{j=2}^{k+1}$ is linearly independent, there is at most one $v_i$ equal to $0.$ If some $v_i$ equals $0$, a simple application of the triangle inequality and the fact that $|v_j| \leq L$ for all $j$ give
$$
|v| \leq \frac{1}{k+1}\sum_{j=1}^{k+1} |v_j| =   \frac{1}{k+1}\sum_{j=1, \, j \neq i}^{k+1} |v_j|  \leq \frac{k}{k+1} L <L.
$$
And if none of the $v_j's¨$ are zero, we use the fact that $|v_j| \leq L$ for all $j$, that $v$ is a strict convex combination of $v_1,\ldots, v_{k+1}$, and that the Euclidean norm $| \cdot |$ is strictly convex in $\R^n \setminus \lbrace 0 \rbrace.$ More precisely, if $m\in \N$, $ \lambda_1, \ldots, \lambda_m \in (0,1)$ are so that $\sum_{j=1}^{m} \lambda_j=1$, and $\lbrace x_j \rbrace_{j=1}^m \subset \R^n \setminus \lbrace 0 \rbrace$ are vectors with $\lbrace x_j-x_1 \rbrace_{j=1}^m$ linearly independent, then
$$
\left | \sum_{j=1}^m \lambda_j x_j \right | < \sum_{j=1}^m \lambda_j |x_j|.
$$
This property is easily shown by first using that $\langle x_i,x_j \rangle < |x_i| |x_j|$ for $i \neq j$ and then an argument of induction. 
\\
Now that we know that $|v|<L$, we have that $|v_j-v|\leq  L + |v|$ by the triangle inequality, which implies that $\lip(c,X) \leq  L+ |v|$, as the construction of $c:X \to \R$ (see \eqref{eq:decomposition_maxfunction_cPX}), gives
$$
c(y) = \max\lbrace \langle v_j-v, y \rangle + b_j \, : \, j=1, \ldots, k+1 \rbrace, \quad y \in X.
$$
\end{proof}

We now extend Lemma \ref{lem:maxfunction_esscoercive} to arbitrary convex and Lipschitz functions in $\R^n.$ In the proof we use the strategy and many arguments from \cite[Theorem 1.11]{AM19APDE}, but with several non-trivial modifications to obtain the right estimate for the norm of the additive linear perturbation $\langle v, \cdot \rangle.$

\begin{thm}\label{thm:decomposition_controlleddirection}
Let $f: \R^n \to \R$ be a non-affine convex Lipschitz function and $L>0$ so that $\lip(f) \leq L.$ Then there exists a subspace $X \subset \R^n$, a function $c: X \to \R$ and a vector $v\in \R^n$ so that
\begin{enumerate}
    \item[\normalfont{(1)}] $f(x) = c(P_X(x)) + \langle v,x \rangle$ for all $x\in \R^n.$
    \item[\normalfont{(2)}] $c: X \to \R$ is convex and coercive.
    \item[\normalfont{(3)}] $|v|<L$.
\end{enumerate}
\end{thm}
 \begin{proof}
First we assume that $f \in C^1(\R^n).$ Because $f$ is not affine, we can find points $y_1, y_2\in\R^n$ with $\nabla f(y_1)\neq \nabla f(y_2)$. By the convexity of $f$, we see that
$$
f(x) \geq C(x):=\max\{f(y_1)+\langle \nabla f(y_1), x-y_1 \rangle,
f(y_2)+\langle \nabla f(y_2), x-y_2 \rangle\}, \quad \text{for all } \, x\in \R^n.
$$
We also know that $|\nabla f (y_1)|, |\nabla f(y_2)| \leq L,$ as $f$ is $L$-Lipschitz. Let then $1 \leq k \leq n$ be the greatest integer so that there exist $v_1,\ldots,v_{k+1} \in \R^n$ with $\lbrace v_j-v_1 \rbrace_{j=2}^{k+1}$ linearly independent and $|v_j| \leq L$ for all $j=1,\ldots,k+1$, $a_1, \ldots, a_{k+1} \in \R$ such that
\begin{equation}\label{eq:induction_maxfunctions}
f(x) \geq  C(x):=  \max \lbrace \langle v_j,x \rangle +a_j \: : \: j=1, \ldots, k+1 \rbrace, \quad \text{for all } \, x\in \Rn.   
\end{equation}
Let $v_1,\ldots,v_{k+1} \in \R^n$, $a_1, \ldots a_{k+1} \in \R$, and $x_0 \in \R^n$ as above and, as in Lemma \ref{lem:maxfunction_esscoercive}, we define $u_j= v_j-v_1$ for $j=2,\ldots, k+1$, $X = \mathrm{span}\lbrace u_j \: : \: j=2,\ldots, k+1 \rbrace$ and $v = \frac{1}{k+1} \sum_{j=1}^{k+1} v_j.$ 

Our next claim is that the function $f-\langle v, \cdot \rangle$ is constant on each affine subspace $y+ X^\perp$, $y\in \R^n$. Indeed, otherwise we can find $w\in X^\perp$, $y\in \R^n$ so that the function $\R \ni t\mapsto (f-\langle v, \cdot \rangle)(y + t w) $ has non-zero derivative at some $t\in \R.$ Therefore $\langle \nabla f(y + t w) - v, w \rangle \neq 0$, and thus $\nabla f(y + t w) - v$ is linearly independent with $\lbrace u_j \}_{j=2}^{k+1}$, as this family forms a basis of $X.$ But we also notice that 
$$
v_1-v = \sum_{j=1}^{k+1} \frac{1}{k+1} (v_1-v_j) \in X,
$$
which implies that the family $\lbrace \nabla f(y + t w) - v_1, v_j-v_1 \rbrace_{j=2}^{k+1}$ consists of $k+2$ linearly independent vectors. Again, since $f$ is $L$-Lipschitz, one has $| \nabla f (y + t w) | \leq L.$ The convexity of $f$ and the properties of \eqref{eq:induction_maxfunctions}, we have that 
$$
f(x) \geq C^*(x):=\max \lbrace  f(y + t w) + \langle \nabla f(y + t w), x-y-t w \rangle, \langle v_j, x \rangle + a_j \: : \: j=2, \ldots, k+1 \rbrace \quad \text{for all } \, x\in \Rn.
$$
This contradicts the maximality of $k$ for which $f$ satisfies \eqref{eq:induction_maxfunctions}. 
\\
Consequently, $f- \langle v, \cdot \rangle$ is a constant function on each affine subspace $y+ X^\perp$, $y\in \R^n.$ Therefore, if we define the convex function $c:X \to \R$ by
$$
c(z)=f(z)-\langle v, z \rangle \quad \text{for all } \, z\in X,
$$
then for every $x\in \R^n$ we can write
$$
(f- \langle v, \cdot \rangle)(x) = (f- \langle v, \cdot \rangle)(P_X(x)) = c(P_X(x)),
$$
thus proving the first statement of the theorem. By the definition of $v$ and Lemma \ref{lem:maxfunction_esscoercive}, we also know that $|v|<L.$ To verify that $c$ is coercive in $X,$ we use Lemma \ref{lem:maxfunction_esscoercive} to write the function $C: \R^n \to \R$ from \eqref{eq:induction_maxfunctions} in the form $C(x) = \widetilde{c}(P_X(x)) + \langle v, x \rangle$, $x\in \R^n,$ where $\widetilde{c} : X \to \R$ is convex and coercive in $X.$ Because $f \geq C$ in $\R^n$, we have the inequalities
$$
c(z) = f(z)- \langle v, z \rangle \geq C(z) - \langle v,z \rangle = \widetilde{c}(z), \quad z\in X.
$$
We conclude that $c$ is coercive in $X$ too, thus proving the theorem for functions $f\in C^1(\R^n).$

\medskip

Now let $f : \R^n \to \R$ be a non-affine convex and $L$-Lipschitz function, but not necessarily of class $C^1(\R^n).$ For a family of non-negative mollifiers $\lbrace \theta_\varepsilon \rbrace_{\varepsilon>0} \subset C_0^\infty(\R^n)$ with $\int_{\R^n} \theta_\varepsilon(y) \ud y = 1$, $\theta_\varepsilon \geq 0$ and $\sop(\theta_\varepsilon) \subset B(0,\varepsilon)$, it is a standard result that the integral convolutions $ g_\varepsilon : = f * \theta_\varepsilon$, $\varepsilon >0,$ define convex and $L$-Lipschitz functions of class $C^\infty(\R^n)$ so that $g_\varepsilon \to f$ uniformly on $\R^n$ as $\varepsilon \to 0^+.$ Thus we can find $g : \R^n \to \R$ of class $C^\infty(\R^n)$, convex and $L$-Lipschitz and so that $f-1\leq g\leq f$ on $\R^n.$ We observe that $g$ cannot be affine, because otherwise the non-affine function $f$ would lie below the affine $g+1$ in all of $\R^n$, a contradiction. Therefore, we can apply the case for $C^1$ functions to $g$ in order to obtain a decomposition
\begin{equation}\label{eq:decompositionapproximating}
    g(x) = c(P_X(x)) + \langle v, x \rangle, \quad x\in \R^n;
\end{equation}
where $X \subset \R^n$ is a subspace, $c : X \to \R$ is convex and coercive and $|v|<L.$ We will first prove that 
\begin{equation}\label{eq:f_affine_orthogonal}
    f(x+w)= f(x) + \langle v, w \rangle \quad \text{for all } \, x\in \R^n, \, w\in X^\perp. 
\end{equation}
Indeed, we define the functions $h(t)=f(x+tw)$ and $\varphi(t)=g(x+tw),$ for $t\in \R$. We have that 
$$
\varphi(t)=c(P_X(x+tw))+ \langle v,x+tw \rangle = c(P_X(x)) + \langle v, x+tw \rangle, 
$$
that is, $\varphi$ is affine with $\varphi'(t)= \langle v,w \rangle$ for all $t\in \R.$ For a similar reason as above, $h$ must be affine with derivative $\langle v,w \rangle$. Let us explicitly write the argument. If $\varphi(t) = a+t\langle v,w \rangle$ with $a\in \R$, we see that for $s\in \R$ and $ \xi \in \partial h(s),$ the inequality $f \leq 1 + g$ gives
$$
a+ 1 + \langle v,w \rangle t = \varphi(t) + 1 \geq h(t) \geq h(s) + \xi(t-s),
$$
 for all $t\in \R.$ This of course implies that $\xi = \langle v,w \rangle$. We have proved that $\partial h(s) = \lbrace \langle v,w \rangle \rbrace$ for all $s\in \R,$ and hence $h$ is affine function with linear form $\langle v,w \rangle.$ This shows \eqref{eq:f_affine_orthogonal}.
\\
We next define the convex function $\widetilde{c} : X \to \R$ by $\widetilde{c}(x) = f(x) - \langle v,x \rangle$ for all $x\in X.$ Using that $f \geq g$ in $\R^n$ and \eqref{eq:decompositionapproximating}, we get the following inequalities for $x\in X:$
$$
\widetilde{c}(x) = f(x) - \langle v,x \rangle \geq g(x)- \langle v,x \rangle = c(x);
$$
where $c:X \to \R$ is coercive. Consequently, $\widetilde{c}$ is coercive in $X$ as well. Also for every $x\in \R^n$, we can use \eqref{eq:f_affine_orthogonal} to write
$$
f(x) = f(P_X(x)) + \langle v, P_{X^\perp}(x) \rangle = \widetilde{c}(P_X(x)) + \langle v, P_X(x) \rangle + \langle v, P_{X^\perp}(x) \rangle  =  \widetilde{c}(P_X(x)) + \langle v, x \rangle.
$$
We have thus shown properties $(1),(2),(3)$ for all non-affine, convex and $L$-Lipschitz functions. 

\medskip

\end{proof}

In the following lemma we collect various relevant properties of decompositions such as those of Theorem \ref{thm:decomposition_controlleddirection}, which were implicitly proven and used in \cite{AM19APDE}; see for example Lemma 4.4 there. Here we provide the detailed proof for the sake of completeness. The properties are true for not necessarily Lipschitz functions.

\begin{lemma}\label{lem:properties_subdifferentials_decompositions}
    Let $f: \R^n \to \R$ be a non-affine convex function so that $f(x) = c(P_X(x)) + \langle v,x \rangle$, $x\in \R^n$, with $X \subset \R^n$ is a subspace, $c :X \to \R $ is convex and coercive in $X,$ and $v\in \R^n$. 
    \\
    Then, for every $x\in \R^n$ and $\eta \in \partial f(x)$, we have $ \eta-v  \in \partial c ( P_X(x)) \subset X.$ In addition, there holds
\begin{equation}\label{eq:formula_Xequalspan_subdifferentials}
X= \mathrm{span} \lbrace \xi_x-\xi_y \: : \: \xi_x\in \partial f (x), \: \xi_y \in \partial f(y),\: x,y\in \R^n \rbrace.    
\end{equation}
Moreover, if $f$ is represented in the form $f(x) = \sup\lbrace f(p) + \langle u_p, x-p \rangle    \, : \, p\in A \rbrace,$ $x\in \R^n$, with $A \subset \R^n$ and $(u_p)_{p\in A} \subset  \R^n$, then 
\begin{equation}\label{eq:maxfunction_Xequalspan_linearparts}
X= \mathrm{span} \lbrace u_p-u_q \: : \: p,q \in A \rbrace. 
\end{equation}
\end{lemma}
\begin{proof}
Given $x\in \R^n$, and $\eta \in \partial f(x)$, assume for the sake of contradiction that $\eta-v \notin X.$ Then we can find $w \in X^\perp$ with $\langle \eta -v , w \rangle =1.$ By the definition of $\partial f(x),$ we can write
$$
\langle \eta , w \rangle \leq f(x+w)-f(x) = c(P_X(x+w))+\langle v, x+w \rangle - c(P_X(x))-\langle v,x \rangle = \langle v,w \rangle,
$$
and therefore $\langle \eta -v , w \rangle \leq 0$, a contradiction. This shows that $\eta -v \in X.$ Now, for every $x\in \R^n$ and $z\in X$ using again that $f= c \circ P_X+ \langle v, \cdot \rangle,$ we can write
\begin{align*}
c(z)-c(P_X(x))= f(z)-\langle v,z \rangle -f(x)+\langle v,x \rangle \geq  \langle \eta -v,z-x \rangle = \langle \eta -v,z-P_X(x) \rangle.
\end{align*}
Consequently, $\eta -v \in \partial c (P_X(x)).$ 

\smallskip

To prove \eqref{eq:formula_Xequalspan_subdifferentials}, we denote by $Z$ the linear span of all the differences of any two subdifferentials of $f.$ The inclusion $Z 	\subset X$ follows from what we have just proved about the subdifferentials of $f$ and $c.$ In order to prove that $X \subseteq Z,$ assume for the sake of contradiction that there are directions $w\in X\setminus \lbrace 0 \rbrace$ with $w \in Z^\perp.$ We fix a point $x_0\in \R^n$ and take some $\xi_t \in \partial f(x_0+tw)$. In particular, we have $\xi_0\in \partial f(x_0).$ Thus $\xi_0-\xi_t\in Z$ for every $t\in \R$ and then
$$
0\leq f(x_0+tw)-f(x_0)-\langle \xi_0, t w \rangle \leq \langle \xi_t-\xi_0, t w \rangle = 0,
$$
that is, $f(x_0+tw) = f(x_0)+ \langle \xi_0, tw \rangle$ for every $t\in \R.$ Using again that $f=c \circ P_X + \langle v,\cdot \rangle$ we get
$$
 c\left( P_X(x_0)+tw \right)=f(x_0+tw) - \langle v,x_0+tw \rangle = f(x_0)+ \langle \xi_0, w \rangle- \langle v,x_0+tw \rangle, \quad t\in \R,
$$
thus showing that $c$ is not coercive along the line $\lbrace P_X(x_0) + t w \rbrace_{t \in \R},$ a contradiction. 

\smallskip

Finally, in order to prove \eqref{eq:maxfunction_Xequalspan_linearparts}, denote $Y= \mathrm{span} \lbrace u_p-u_q \: : \: p,q \in A \rbrace .$ It is immediate that $u_p \in \partial f(p)$ for each $p\in A$, and so the inclusion $Y \subset X$ is a consequence of \eqref{eq:formula_Xequalspan_subdifferentials}. Conversely, if we had $X\neq Y$, we could find $w\in (X \cap Y^{\perp}) \setminus\{0\}$. We fix a point $q\in A$ and write, for each $t\in \R,$
\begin{align*}
0 \leq  f(q+tw)  -f(q) -\langle u_q, tw\rangle &  = \sup_{p\in E}\{ f(p)+\langle u_p, q+tw - p \rangle-f(q) -\langle u_q, tw \rangle \}\\
 & = \sup_{p\in E}\{ f(p)+\langle u_p, q- p \rangle-f(q) +\langle u_p-u_q, tw \rangle \} \\
 & =  \sup_{p\in E}\{ f(p)+\langle u_p, q- p \rangle-f(q)  \rangle \} 
 \leq 0,
\end{align*} 
after using that $u_p \in \partial f (p)$ in the last inequality. Therefore, $f(q+tw) =  f(q) + \langle u_q, w \rangle t$ for every $t\in \R.$ Repeating an argument almost identical to that of the proof of \eqref{eq:formula_Xequalspan_subdifferentials}, we get that $c$ is not coercive along the the line $\lbrace P_X(p) + t w \rbrace_{t \in \R},$ a contradiction. We conclude that $X=Y$. 
\end{proof}

\section{Proof of Theorem \ref{thm:maintheorem}}\label{sect:proofmain}

Assume that $E \subset \R^n$ is arbitrary, that $(f,G) : E \to \R \times \R^n$ is a $1$-jet on $E$ with $G$ non-constant, and that $Y,X \subset \R^n$ are linear subspaces so that $f,G,Y,X$ satisfy the conditions \ref{condition:convexity}, \ref{condition:subspacecontained}, \ref{condition:existencecones}, \ref{condition:cornersorthogonal} on the set $E.$ We will also denote
\begin{equation}\label{eq:notation_L_supremumG}
    L:= \sup_{z\in E} |G(z)|.
\end{equation}
throughout this section.

\subsection{The first minimal function}\label{subsec:firstminimal}
We start by defining
\begin{equation}\label{eq:definition_m}
    m(x) = \sup_{y\in E} \lbrace f(y)+ \langle G(y), x-y \rangle \rbrace, \quad x\in \R^n.
\end{equation}
First of all notice that if we fix $y_0 \in E,$ \eqref{eq:notation_L_supremumG} and condition \ref{condition:convexity} can be used to obtain, for all $y\in E$ and $x \in \R^n$:
\begin{align*}
  f(y_0) + L |y_0-x| & \geq f(y_0) + \langle G(y), y_0-x \rangle \\
  & \geq f(y) +\langle G(y), y_0-y \rangle  +\langle G(y), y_0-x \rangle = f(y)+ \langle G(y), x-y \rangle.  
\end{align*}
Therefore the supremum \eqref{eq:definition_m} defining $m(x)$ is finite for all $x\in \R^n.$ And since each function $x\mapsto f(y)+ \langle G(y), x-y \rangle$ is convex and $L$-Lipschitz, $m$ is convex and $L$-Lipschitz as well. Moreover, if we apply condition \ref{condition:convexity}, we see that $m(x)=f(x)$ for all $x \in \R^n.$ Therefore, an alternate way to write $m$ is via the supremum
\begin{equation}\label{eq:definition_m_withminplaceoff}
    m(x) = \sup_{y\in E} \lbrace m(y)+ \langle G(y), x-y \rangle \rbrace, \quad x\in \R^n.
\end{equation}
In particular, we deduce that $G(y) \in \partial m(y)$ for all $y\in E,$ and since $G$ is non-constant, we have that $m$ is not affine in $\R^n.$ 

\smallskip

Hence we can apply Theorem \ref{thm:decomposition_controlleddirection} to find a linear subspace $Z \subset \R^n$, convex and coercive function $c : Z \to \R$ and a vector $v\in \R^n$ so that $|v|<L$ and
$$
m(x) = c \left( P_{Z}(x) \right) + \langle v, x \rangle, \quad x\in \R^n;
$$
where $P_Z: \R^n \to Z$ is the orthogonal projection onto $Z.$ Moreover, in view of \eqref{eq:definition_m_withminplaceoff}, we can apply formula \eqref{eq:maxfunction_Xequalspan_linearparts} from Lemma \ref{lem:properties_subdifferentials_decompositions} to infer  
$Z=Y= \textrm{span} \lbrace G(x)-G(y) \: : \: x,y \in E \rbrace
$ and therefore we have
\begin{equation}\label{firstdecompositionminimal}
m (x)= c( P_{Y}(x)) + \langle v, x \rangle, \quad x\in \R^n; \quad c: Y \to \R \, \text{ convex and coercive; } \: v\in \R^n, \, |v|<L.
\end{equation}
We also observe that, since $G(x) \in \partial m (x)$, $x\in E$, Lemma \ref{lem:properties_subdifferentials_decompositions} gives
\begin{equation}\label{eq:Y_span_G-v}
    Y= \mathrm{span}\lbrace G(x)- v \, : \, x\in E \rbrace.
\end{equation}
As a convex coercive function, $c:Y \to \R$ satisfies the following property:
\begin{equation}\label{eq:function_c_cone}
\text{there exist } \: 0 < \alpha \leq L + |v|, \, \beta \in \R \: \text{ so that } \: c(y) \geq a |y|  + b \: \text{ for every } \: y \in Y.
\end{equation}
The simple proof of this fact is as follows. There must exist $a > 0$ so that $c(y) \geq a |y|$ for large enough $|y|$, $y\in Y,$ as, otherwise, we could find a sequence $(y_j)_j \subset Y$ with
$$
|y_j| \geq j \quad \text{and} \quad c(y_j) \leq \frac{1}{j}|y_j| \quad \text{for every} \quad \ell \in \N,
$$
and the convexity of $c$ would lead us to
$$
c \left( \frac{j}{|y_j|} y_j \right) \leq  \frac{j}{|y_j|}c (y_j)+\left( 1- \frac{j}{|y_j|} \right) c (0)  \leq 1+\left( 1- \frac{j}{|y_j|} \right) c (0) \leq 1+ c(0),
$$
contradicting that $c$ is coercive. Thus there are $a,r>0$ with $c(y) \geq a |y|$ for all $|y| \geq r. $ It is easy to see that setting 
$
b =\min  \Big \lbrace \inf_{B(0,r)} c - \alpha r, 0 \Big \rbrace,
$ 
we have $c(y) \geq a |y| + b$ for every $y\in Y.$ The upper bound $ a \leq L + |v|$ follows from the fact that $\lip(c) \leq L + |v|$, which is a consequence of \eqref{firstdecompositionminimal}.

\subsection{New data preserving the Lipschitz constant}\label{subsection:defining_new_data} Our next objective is to extend the definition of the jet $(f,G) : E \to \R \times \R^n$ to $\dim(X)- \dim(Y)$ new points, so that the new jet $(f^*,G^*) : E^* \to \R \times \R^n$ satisfies the necessary properties for a $C^1$ convex extension, with $X=\mathrm{span}\lbrace G^*(x)-G^*(y) \, : \, x,y \in E^* \rbrace$ and so that $G^*$ does not increase the supremum of $G.$ 

\smallskip

Denote $k = \dim Y$ and $d= \dim X$, and $w_1,\ldots, w_{d-k} \in Y ^{\perp} \cap X, \: \varepsilon \in (0,1), \:p_1, \ldots, p_{d-k}$ and the cones $V_1,\ldots, V_{d-k}$ as in condition \ref{condition:existencecones}. Also, let $a,b$ the constants from \eqref{eq:function_c_cone}. Because $|v| < L$ (see \eqref{firstdecompositionminimal}), we can find 
\begin{equation}\label{eq:parameter_theta}
0<\theta <1 \: \text{ such that } \: \theta \cdot a \cdot \varepsilon < \frac{L-|v|}{2}.
\end{equation}
We next choose a parameter $T>0$ large enough so that 
\begin{equation}\label{eq:choiceT_distinct_qj}
    T > \max_{1 \leq i < j \leq d-k} \frac{|p_i-p_j|}{|w_i-w_j|},
\end{equation}
\begin{equation}\label{eq:parameter_T_cones}
  (\theta \cdot \varepsilon \cdot a ) \, T \geq 2 - b + \max_{j=1,\ldots, d-k} \lbrace       c(P_Y(p_j) ) + a |P_Y(p_j)| \rbrace   , \quad \text{and}
\end{equation}
\begin{equation}\label{eq:parameter_T_newpoints_compatible}
\frac{\varepsilon \cdot a}{2} \cdot T \min_{1 \leq i <j \leq d-k}   |w_i-w_j|^2   \geq 1 + \max_{1 \leq i,j \leq d-k} \lbrace | c(P_Y(p_j))-c(P_Y(p_i)) + \varepsilon \cdot a \cdot \langle w_j, p_i-p_j \rangle | \rbrace.
\end{equation}
We define new points and $1$-jet data there as follows
\begin{equation}\label{eq:definitionnewpoints}
q_j= p_j + T w_j, \quad f^*(q_j) = m(q_j)+1, \quad G^*(q_j)= v + (\theta \cdot \varepsilon \cdot a) \cdot w_j, \quad j=1, \ldots, d-k .
\end{equation}
Observe that condition \eqref{eq:choiceT_distinct_qj} assures that $q_i\neq q_j$ for $i\neq j.$ Also, it is clear that $q_j \in V_j,$ as $w_j \in Y^\perp.$ In particular, condition \ref{condition:existencecones} guarantees that $q_j \notin \overline{E}$, for all $j=1, \ldots, d-k.$

\medskip

As in \cite[Lemma 6.2]{AM19APDE}, we next show that the new $1$-jet data is compatible with a $C^1$ convex extension problem. This means that the old $(f,G)$ and the new data $(f^*,G^*)$ jointly satisfy the natural condition of convexity \ref{condition:convexity}. In fact, we show this condition with positive lower bounds, which will allow us to extend the validity of condition \ref{condition:cornersorthogonal} to the new jet as well.

\begin{lemma} \label{lem:newdata_compatible}
For the data \eqref{eq:definitionnewpoints}, the following hold.
\begin{enumerate}
\item[$(1)$] $f^*(q_j)-f(x)-\langle G(x), q_j-x \rangle \geq 1$ for every $x\in E,\: j=1,\ldots, d-k.$
\item[$(2)$] $f(x)-f^*(q_j)-\langle G^*(q_j), x-q_j \rangle \geq 1$ for every $x\in E,\: j=1,\ldots, d-k.$
\item[$(3)$] $f^*(q_i)-f^*(q_j)-\langle G^*(q_j), q_i-q_j \rangle \geq 1$ for every $1 \leq i \neq j\leq  d-k.$
\end{enumerate}
\end{lemma}
\begin{proof}
\item[] $(1)$ Recall the definition of $m$ in \eqref{eq:definition_m} and that $f^*(q_j) = m(q_j)+1.$ This gives the inequalities
$$
f^*(q_j)-f(x)-\langle G(x), q_j-x \rangle = m(q_j)-f(x)-\langle G(x), q_j-x \rangle + 1 \geq 1,
$$
for all $x\in E, \: j=1, \ldots, d-k.$
\item[] $(2)$ We fix $x \in E$ and $j=1, \ldots, d-k.$ By formula \eqref{firstdecompositionminimal} for $m$ and $w_j \in Y^\perp$, we easily get 
$$
m(q_j) = m(p_j)+ \langle v, q_j-p_j \rangle.
$$
This identity and \eqref{firstdecompositionminimal} again give
\begin{equation}\label{eq:m(x)minusm(qj)}
    m(x)- m(q_j) = m(x) - m(p_j) - \langle v, q_j-p_j \rangle = c(P_Y(x)) - c(P_Y(p_j)) + \langle v, x-q_j \rangle.
\end{equation}
Looking at the definitions of $f^*(q_j)$ and $G^*(q_j)$ from \eqref{eq:definitionnewpoints}, recalling that $m=f$ on $E$ (as we observed in Section \ref{subsec:firstminimal}), and using \eqref{eq:m(x)minusm(qj)}, we arrive at 
\begin{align*}
f(x)-f^*(q_j)& -\langle G^*(q_j), x-q_j \rangle  = c(P_Y(x)) - c(P_Y(p_j)) + \langle v-G^*(q_j), x-q_j \rangle  -1 \\
& =  c(P_Y(x)) - c(P_Y(p_j)) - \theta \cdot \varepsilon \cdot a \langle w_j, x-q_j \rangle  -1 \\
& =  c(P_Y(x)) - c(P_Y(p_j)) - \theta \cdot \varepsilon \cdot a \langle w_j, x-p_j \rangle +(\theta \cdot \varepsilon \cdot a ) T   -1 .
\end{align*}
Using \eqref{eq:parameter_T_cones} the last term is not smaller than
$$
c(P_Y(x))    - b +1    - \theta \cdot \varepsilon \cdot a \langle w_j, x-p_j \rangle  .
$$
And using \eqref{eq:function_c_cone}, the last term is in turn larger than or equal to
$$
a | P_Y(x)|   + a |P_Y(p_j)|  +1    - \theta \cdot \varepsilon \cdot a \langle w_j, x-p_j \rangle   \geq a | P_Y(x-p_j)| + 1 -\theta \cdot \varepsilon \cdot a \langle   w_j , x-p_j \rangle  .
$$
To estimate the last term from below, we argue as follows. In the case where $\langle   w_j , x-p_j \rangle \leq 0,$ the last term is obviously not smaller than $1.$ And if $\langle   w_j , x-p_j \rangle >0$, we condition \ref{condition:existencecones} (namely, that $x\notin V_j$) to estimate
$$
a | P_Y(x-p_j)| + 1 -\theta \cdot \varepsilon \cdot a \langle   w_j , x-p_j \rangle \geq a \langle   w_j , x-p_j \rangle  + 1 -\theta \cdot \varepsilon \cdot a \langle   w_j , x-p_j \rangle \geq 1.
$$
Following the chain of inequalities, we arrive at the desired estimate. 

\item[] $(3)$ Let $1 \leq i \neq j\leq d-k$. By the definition of $f^*$ in \eqref{eq:definitionnewpoints}, the formula \eqref{firstdecompositionminimal}, and the fact $w_i,w_j \in Y^\perp$, we easily get
$$
f^*(q_i)-f^*(q_j) =  c (P_Y(p_i))-c (P_Y(p_j))+\langle v, q_i-q_j \rangle.
$$
This identity and the definitions of $q_i,q_j$ and $G^*$ in \eqref{eq:definitionnewpoints} lead us to 
\begin{align*}
f^*(q_i)-f^*(q_j)& -\langle G^*(q_j), q_i-q_j \rangle = c (P_Y(p_i))-c (P_Y(p_j))+\langle v, q_i-q_j \rangle- \langle v+ \varepsilon \cdot \alpha \cdot  w_j, q_i-q_j \rangle \\
& = c (P_Y(p_i))-c (P_Y(p_j))- \varepsilon \cdot a \langle w_j, p_i-p_j+ T(w_i-w_j) \rangle \\
& = c (P_Y(p_i))-c (P_Y(p_j))- \varepsilon \cdot a \langle w_j, p_i-p_j \rangle + \varepsilon \cdot a \cdot T \left( 1- \langle w_i, w_j \rangle \right) \\
& = c (P_Y(p_i))-c (P_Y(p_j))- \varepsilon \: a \langle w_j, p_i-p_j \rangle + \frac{\varepsilon \cdot a}{2} \cdot T |w_i-w_j|^2 .
\end{align*}
We used in the last equality that $|w_i|=|w_j|=1.$ The last term is clearly greater than or equal to $1$ by virtue of \eqref{eq:choiceT_distinct_qj}. 
\end{proof}

\subsection{Extending to finitely-many points}
With the initial data $E, f,G$ and the new ones \eqref{eq:definitionnewpoints}, we define a extended jet $E^*,f^*,G^*$:
\begin{equation}\label{eq:def_newjet}
E^*:= E \cup \lbrace q_1, \ldots, q_{d-k} \rbrace, \quad (f^*,G^*):= \begin{cases}  
 (f,G)  & \text{on } \, E \\
   \text{as in } \eqref{eq:definitionnewpoints} & \text{on } \, \lbrace q_1, \ldots, q_{d-k} \rbrace.
\end{cases}   
\end{equation}
The following lemma shows crucial properties of the new jet $(E^*,f^*,G^*).$ For our purpose, it is particularly important that $G^*$ preserves the supremum of $G.$

\begin{lemma}\label{lemm:properties_newjets}
The jet defined in \eqref{eq:def_newjet} has the following properties.
\begin{enumerate}
\item[$(1)$] $X= \mathrm{span} \lbrace G^*(x)-G^*(y) \: :\: x,y\in E^*\rbrace.$ 
\item[$(2)$] $\sup_{x\in E^*} |G^*(x)|  = \sup_{x\in E} |G(x)| = L$. 
\item[$(3)$] Condition \ref{condition:convexity} holds for $(E^*,f^*,G^*)$ in place of $(E,f,G)$: $G^*$ is continuous on $E^*$ and $f^*(x) \geq f^*(y)+\langle G^*(y),x-y \rangle$ for all $x,y \in E^*.$ 
\item[$(4)$] Condition \ref{condition:cornersorthogonal} holds for $(E^*,f^*,G^*)$ in place of $(E,f,G)$: for any two sequences $(x_j)_j, (z_j)_j \subset E^*$ such that $(P_X(x_j))_j$ is bounded, one has
$$
\lim_{j \to\infty} \left( f^*(x_j)-f^*(z_j)- \langle G^*(z_j), x_j-z_j \rangle \right) = 0 \implies \lim_{j \to\infty} | G^*(x_j)-G^*(z_j) | = 0.
$$
\end{enumerate}
\end{lemma}
\begin{proof}
\item[] $(1)$ This is an easy consequence of \eqref{eq:Y_span_G-v}, the definitions of $(G^*(q_j))_{j=1}^{d-k}$ in \eqref{eq:definitionnewpoints} and the fact that the vectors $w_1, \ldots, w_{d-k} $ are linearly independent and contained in $Y^\perp.$

\medskip

\item[] $(2)$ Recall that $L$ is the supremum of $G$ on $E.$ Also, $G^*(q_j)=  v + (\theta \cdot \varepsilon \cdot a) \cdot w_j$ (see \eqref{eq:definitionnewpoints}); where $\theta$ is as in \eqref{eq:parameter_theta}. Therefore $|G^*(q_j)| \leq \frac{L+ |v|}{2}$ for all $j=1,\ldots,d-k.$ But also $|v|<L,$ as per \eqref{firstdecompositionminimal}. We conclude that $|G^*(x)| \leq L$ for all $x\in E^*.$

\medskip

\item[] $(3)$ As we observe right after \eqref{eq:definitionnewpoints}, the points $q_1, \ldots, q_{d-k}$ are distinct and outside $\overline{E}.$ Therefore $G^*$ is continuous in $E^*$, as $G$ is in $E$. Moreover, condition \ref{condition:convexity} for $(f,G)$ and Lemma \ref{lem:newdata_compatible} imply that
$$
f^*(x) \geq f^*(y)+ \langle G^*(y), x-y \rangle  \quad \text{for all} \quad x,y \in E^*.
$$

\medskip

\item[] $(4)$ 
According to Lemma \ref{lem:newdata_compatible}, the limit
$$
\lim_{j\to\infty} \left( f^*(x_j)-f^*(z_j)- \langle G^*(z_j), x_j-z_j \rangle \right)=0
$$ 
implies that either the sequences $(x_j)_j, (z_j)_j$ are stationary and equal to some $q_i,$ or else they are entirely contained in $E$ for large enough $j.$ In the former case the conclusion is trivial, and in the latter we get that 
$$
\lim_{j\to\infty} |G^*(x_j)-G^*(z_j)|= \lim_{j\to\infty} |G(x_j)-G(z_j)| = 0
$$
from condition \ref{condition:cornersorthogonal}. 
\end{proof}

\subsection{The second minimal function}\label{subsect:secondminimal}

We define 

\begin{equation}\label{eq:definition_secondminimal}
    m^*(x) = \sup _{y\in E^*} \lbrace f^*(y)+\langle G^*(y), x-y \rangle \rbrace, \quad x\in \R^n. 
\end{equation}

Bearing in mind properties $(2)$ and $(3)$ from Lemma \ref{lemm:properties_newjets}, we can use the same arguments as those for $m$ in Section \ref{subsec:firstminimal} to derive that $m^*$ is well-defined, $L$-Lipschitz and convex, that $m^*(x)=f^*(x)$ and $G^*(x) \in \partial m^*(x)$ for all $x\in E^*.$

Applying Theorem \ref{thm:decomposition_controlleddirection}, we can write $m^*$ as
\begin{equation}\label{eq:decomposition_minimal*}
m^* (x)= c^*( P_{X}(x)) + \langle v^*, x \rangle, \quad x\in \R^n; \quad c^*: X \to \R \, \text{ convex and coercive; } \: v^*\in \R^n, \, |v^*|<L.
\end{equation}
Note that we have applied Theorem \ref{thm:decomposition_controlleddirection} in combination with formula \eqref{eq:formulas_subspace_vector_maxfunctions} of Lemma \ref{lem:properties_subdifferentials_decompositions} and also Lemma \ref{lemm:properties_newjets}(1). 

\smallskip

Finally, property $(4)$ of Lemma \ref{lemm:properties_newjets} permits to obtain the following important lemma concerning the differentiability of $c^*$, whose proof is identical to that of \cite[Lemma 4.6]{AM19APDE}.

\begin{lemma} \label{lem:differentiabiltyclosurec^*lipschitz}
The function $c^*$ is differentiable on $\overline{P_X(E^*)}$, and $ \nabla c^* (P_X(x))= G^*(x)-v^*$ for all $x\in E^*$.
\end{lemma}

\subsection{The first $C^1$ convex extension}\label{subsect:notsharpextension}

We are going to apply the following lemma of extension to the function $c^*.$

\begin{lemma}\label{extensionconvexcoercivelipschitz}
Let $h:X \to \R$ be a convex and coercive function such that $h$ is differentiable on a closed subset $A$ of $X.$ There exists $H \in C^1(X)$ convex and coercive such that $H=h$ and $\nabla H = \nabla h$ on $A.$
\end{lemma}
\begin{proof}
See \cite[Lemma 6.5]{AM19APDE}.
\end{proof}

Since $c^*:X \to \R$ is convex and coercive, Lemmas \ref{lem:differentiabiltyclosurec^*lipschitz} and \ref{extensionconvexcoercivelipschitz} give a convex and coercive function $H \in C^1(X)$ with $(H, \nabla H)=(c^*, \nabla c^*)$ on $\overline{P_X(E^*)}$. We define
\begin{equation}\label{eq:def_extension_notsharp}
    \widetilde{F}(x) = H (P_X(x)) + \langle v^*, x \rangle, \quad x\in \R^n. 
\end{equation}
The function $\widetilde{F}$ is $C^1(\R^n)$ and convex, because $H:X \to \R$ is convex and of class $C^1(X)$. Since $H(y)=c^*(y)$ for $y\in P_X(E^*)$, the formula \eqref{eq:decomposition_minimal*} tells us that, whenever $x\in E^*,$
$$
\widetilde{F}(x)= H(P_X(x))+\langle v^*, x \rangle = c^* (P_X(x))+\langle v^*, x \rangle=m^*(x)=f^*(x).
$$
And from Lemma \ref{lem:differentiabiltyclosurec^*lipschitz} and again \eqref{eq:decomposition_minimal*}, we see that for all $x\in E^*$
$$
\nabla \widetilde{F}(x)= \nabla H (P_X(x)) + v^* = G^*(x)-v^* + v^* = G^*(x).
$$

\subsection{The sharp $C^1$ convex extension}\label{subsect:finalextension} We finally define the function
\begin{equation}\label{eq:final_extension_F}
    F(x) := \inf\lbrace \widetilde{F}(y) + L |x-y| \: : \: y\in \R^n \rbrace  \quad \text{ for all } \:  x\in \R^n.
\end{equation}

In this section we prove that $F$ is the desired extension of Theorem \ref{thm:maintheorem}. We will in fact prove in the next lemma that $(F, \nabla F)=(f^*,G^*)$ on the extended set $E^*$, which in turn implies $(F, \nabla F) = (f,G)$ on $E$, as per \eqref{eq:def_newjet}.

\begin{lemma}\label{lem:firstproperties_finalextension}
 The formula \eqref{eq:final_extension_F} defines a convex and Lipschitz function $F: \R^n \to \R$ with $\lip(F) \leq L$ and $m^* \leq F \leq \widetilde{F}$ in $\R^n.$ In particular, $F=f^*$ on $E$ and $F$ is differentiable on $E$ with $\nabla F(y) = G^*(y) $ for all $y\in E^*.$
\end{lemma}
\begin{proof}
As we have shown in Section \ref{subsect:notsharpextension}, $(\widetilde{F}, \nabla \widetilde{F})=(f^*,G^*)$ on $E^*$, and so, by the convexity of $\widetilde{F}$, we have that
$$
\widetilde{F}(x) \geq \widetilde{F}(y) + \langle \nabla \widetilde{F}(y), x-y \rangle = f^*(y)+ \langle G^*(y), x-y \rangle 
$$
for every $x,y\in E^*$. By formula \eqref{eq:definition_secondminimal}, this inequality gives that $\widetilde{F} \geq m^*$ in $\R^n$. As pointed out in \eqref{eq:definition_secondminimal}, $m^*$ is $L$-Lipschitz, and therefore
$$
\widetilde{F} (y) + L |x-y| \geq m^*(y) + L |x-y| \geq m^*(x)
$$
for all $x,y\in \R^n.$ This implies that $F$ is well-defined, $F \geq m^*$ in $\R^n$, and $F$ is $L$-Lipschitz, as finite infimum of $L$-Lipschitz functions. Using the convexity of $\widetilde{F},$ a standard argument for inf-convolutions shows that $F$ is convex as well. We include its proof for the completeness. Given points $x,z\in \R^n ,$ $\lambda \in [0,1],$ and $\varepsilon >0,$ we can $y_x, y_z \in \R^n$ with
$$
F(x) \geq \widetilde{F}(y_x) + L|x-y_x| - \varepsilon, \quad F(z) \geq \widetilde{F}(y_z)  + L|z-y_z| -\varepsilon.
$$
By the definition of $F(\lambda  x + (1-\lambda)  z)$, the convexity $F$ in $X$ and the choices of $y_x$ and $y_z$ lead us to the following chain of inequalities.
\begin{align*}
F(\lambda  x + (1-\lambda) z) & \leq \widetilde{F}(\lambda y_x + (1-\lambda) y_z ) + L | \lambda  x + (1-\lambda)  z- (\lambda  y_x + (1-\lambda) y_z) |\\
& \leq \lambda \widetilde{F}(y_x) + (1-\lambda) \widetilde{F}(y_z) + L \lambda |x-y_x| + L (1-\lambda) |z-y_z| \\
& \leq \lambda  ( \widetilde{F}(y_x) + L|x-y_x |  ) + (1-\lambda) ( \widetilde{F}(y_z) + L|z-y_z|  ) \\
& \leq \lambda (F(x)+\varepsilon) + (1-\lambda) ( F(z) + \varepsilon ) \\
& = \lambda F(x) + (1-\lambda) F(z) + \varepsilon.
\end{align*}
Letting $\varepsilon \to 0$, the above shows that $F$ is convex. 

\smallskip

The inequality $F \leq \widetilde{F}$ is obvious from \eqref{eq:final_extension_F}. Let us prove that $(F, \nabla F)=(f^*,G^*)$ on $E^*$. We showed in Section \ref{subsect:secondminimal} that $m^*=f^*$ on $E^*$ and $G^*(y) \in \partial m^*(y)$ for all $y\in E^*$, and in Section \ref{subsect:notsharpextension} that $(\widetilde{F}, \nabla \widetilde{F} ) =(f^*,G^*)$ on $E^*. $ As $m^* \leq F \leq \widetilde{F},$ this immediately implies that $F=f^*$ on $E^*$. Moreover, for $y\in E^*$, and $x\in \R^n$, we have
\begin{align*}
0 & \leq  \frac{m^*(x)-f^*(y)- \langle G^*(y), x-y \rangle}{|x-y|} \\
& \leq \frac{F(x)-f^*(y)-\langle G^*(y), x-y \rangle}{|x-y|} \\
& \leq \frac{\widetilde{F}(x)-f^*(y)-\langle G^*(y), x-y \rangle}{|x-y|} ,
\end{align*}
where the last term tends to $0$ as $|x-y| \to 0^+.$ By the convexity of $F,$ we may conclude that $F$ is differentiable at $y$ and $\nabla F(y)=G^*(y)$ for each $y\in  E^*.$
\end{proof}

We next prove that $F$ has the desired global behavior. Recall that $X_F$ is the unique linear subspace of $\R^n$ so that $F$ can be written as $F(x) = C(P_{X_F}(x)) +  \langle u,x \rangle$, $x\in \R^n$, with $C: X_F \to \R$ convex and coercive in $X_F$ and $u \in \R^n.$

\begin{lemma}\label{lem:properties_coerc_finalextension}
The function $F: \R^n \to \R$ satisfies $X_F=X.$
\end{lemma}
\begin{proof}
Let us denote $Z= X_F$. By Theorem \ref{thm:decomposition_controlleddirection} we can write $F(x) = C(P_{Z}(x)) +  \langle u,x \rangle$, $x\in \R^n$, with $C: Z \to \R$ convex and coercive in $X_F$ and $u \in \R^n.$
\\
Now, for every $w\in X^\perp$, the function $\widetilde{F}$ is affine along the line generated by $w$ by virtue of formula \eqref{eq:def_extension_notsharp}. But since $F \leq \widetilde{F}$ in $\R^n$ we have that
$$
C(t P_Z(w) ) =F(tw)- \langle u, w \rangle t \leq \widetilde{F}(tw)- \langle u, w \rangle t , \quad t\in \R.
$$
Since $C$ is coercive in $Z$ and $\R \ni t\mapsto \widetilde{F}(tw)$ is affine by the above, this inequality imply that necessarily $P_Z(w)=0,$ that is $w\in Z^\perp.$ We have therefore shown the inclusion $Z \subset X.$

\smallskip

The reverse inclusion $X \subset Z$ can be shown via the inequality $m^* \leq F$ and a similar argument as above, or using formulas \eqref{eq:formulas_subspace_vector_maxfunctions} and \eqref{eq:formula_Xequalspan_subdifferentials} of Lemma \ref{lem:properties_subdifferentials_decompositions} and the fact that $\nabla F = G^*$ on $E^*$: 
$$
X = \mathrm{span}\lbrace G^*(x)-G^*(y) \, : \, x,y\in E^* \rbrace \subseteq \mathrm{span}\lbrace \xi- \eta \: : \: \xi \in \partial F(x), \, \eta \in \partial F(y), \, x,y\in \R^n \rbrace =Z.  
$$
\end{proof}

To complete the proof of Theorem \ref{thm:maintheorem}, it remains to prove that $F \in C^1(\R^n)$. 

\begin{lemma}\label{lem:finalextension_C1}
The function $F$ is of class $C^1(\R^n)$.
\end{lemma}
\begin{proof}
By the convexity of $F,$ it suffices to show that $F$ is differentiable at every $x\in \R^n.$ Moreover, the differentiability of $F$ at $x_0 \in \R^n$ is equivalent to showing that
\begin{equation}\label{eq:criteriaDiff_forC1conv}
\lim_{h \to 0} \frac{F(x_0+h) + F(x_0-h)-2F(x_0)}{|h|}=0.
\end{equation}
To prove \eqref{eq:criteriaDiff_forC1conv}, let $(h_k)_k \subset \R^n \setminus \lbrace 0 \rbrace$ be a sequence with $|h_k| \downarrow 0.$ By the definition of $F(x_0),$ we can find a sequence $(y_k)_k \subset \R^n$ so that 
\begin{equation}\label{eq:approximation_infimum_F(x0)}
 \widetilde{F}(y_k) + L |x_0-y_k| \leq F(x_0) + \frac{|h_k|}{k}, \quad k\in \N.
\end{equation}
In particular, using \eqref{eq:def_extension_notsharp} we have 
$$
H(P_X(y_k)) + \langle v^*, y_k \rangle + L|x_0-y_k| \leq F(x_0) + \frac{|h_k|}{k} \, \text{ for all } \, k;
$$
where $|v^*| <L$ and $H:X \to \R$ is convex, coercive and $C^1(X).$ It then follows that
\begin{align*}
H(P_X(y_k)) -L |y_0| & \leq H(P_X(y_k)) - L |y_k| + L |x_0-y_k| \\ 
& \leq H(P_X(y_k)) + \langle v^*, y_k \rangle + L|x_0-y_k| \leq F(x_0) + \frac{|h_k|}{k}.
\end{align*}
The coercivity of $H$ implies that the sequence $(P_X(y_k))_k$ is bounded. The definitions of $F(x_0 \pm h_k)$ give
$$
F(x_0 + h_k) \leq \widetilde{F}(y_k + h_k) + L |x_0 + h_k - (y_k + h_k) | = H(P_X(y_k+h_k)) + \langle v^*, y_k + h_k \rangle + L |x_0-y_k| ,
$$
$$
F(x_0 - h_k) \leq \widetilde{F}(y_k - h_k) + L |x_0 - h_k - (y_k - h_k) | = H(P_X(y_k-h_k)) + \langle v^*, y_k - h_k \rangle + L |x_0-y_k| .
$$
In combination with \eqref{eq:approximation_infimum_F(x0)}, these inequalities yield the estimates
\begin{align*}
F(x_0+h_k) & +  F(x_0-h_k)-2F(x_0) \\
& \leq H(P_X(y_k+h_k)) + H(P_X(y_k-h_k)) + 2 \langle v^*, y_k \rangle +2 L |x_0-y_k| - 2F(x_0) \\
& \leq H(P_X(y_k+h_k)) + H(P_X(y_k-h_k)) + 2 \langle v^*, y_k \rangle - 2 \widetilde{F}(y_k) + \frac{2|h_k|}{k} \\
& \leq H(P_X(y_k+h_k)) + H(P_X(y_k-h_k)) - 2 H(P_X(y_k)) + \frac{2|h_k|}{k}.
\end{align*}
Writing $z_k = P_X(y_k)$ and $\xi_k = P_X(h_k)$, we have $|h_k| \leq |\xi_k|$ and the above inequalities yield
\begin{equation}\label{eq:estimate_limit_F_H}
\frac{F(x_0+h_k)   +  F(x_0-h_k)-2F(x_0)}{|h_k|} \leq  \frac{H(z_k+\xi_k) + H(z_k-\xi_k) - 2 H(z_k)}{|\xi_k|} + \frac{2}{k}, \quad k \in \N,
\end{equation}
where $(z_k)_k \subset X$ is bounded and $|\xi_k| \to 0.$ We claim that 
\begin{equation}\label{lem:limit_H_equal0}
 \lim_{k \to \infty}    \frac{H(z_k+\xi_k) + H(z_k-\xi_k) - 2 H(z_k)}{|\xi_k|} = 0.
\end{equation}
Indeed, assume, for the sake of contradiction, that there is $\varepsilon>0$ and subsequences of $(z_k)_k$ and $(\xi_k)_k$ (which we keep denoting by $(z_k)_k$ and $(\xi_k)_k$) so that
$$
\frac{H(z_k+\xi_k) + H(z_k-\xi_k) - 2 H(z_k)}{|\xi_k|} \geq \varepsilon \: \text{ for all } \: k \in \N.
$$
Passing to a further subsequence we may also assume that $(z_k)_k \to z_0 \in X.$ By the convexity and differentiability of $H$ in $X$, we have the inequalities
\begin{align*}
  H(z_k) \geq  H(z_k + \xi_k) - \langle \nabla H(z_k+ \xi_k), \xi_k \rangle \\
H(z_k) \geq  H(z_k - \xi_k) + \langle \nabla H(z_k- \xi_k), \xi_k \rangle, 
\end{align*}
which in combination with the previous $\varepsilon$-lower bound yield
$$
\varepsilon \leq \frac{\langle \nabla H(z_k + \xi_k)- \nabla H(z_k-\xi_k), \xi_k \rangle}{|\xi_k|} \leq | \nabla H(z_k + \xi_k)- \nabla H(z_k-\xi_k)|, \quad k\in \N.
$$
Since $H \in C^1(X)$, the gradient $\nabla H  :X \to X$ is continuous, and since $(z_k)_k$ converges to $z_0,$ and $(\xi_k)_k \to 0,$ the last term converges to $0$, as $k \to \infty,$ a contradiction. This shows that \eqref{eq:estimate_limit_F_H} holds, which, in combination with \eqref{eq:estimate_limit_F_H}, proves \eqref{eq:criteriaDiff_forC1conv}, as desired. 
\end{proof}

Concerning the proof of Theorem \ref{thm:maintheorem}, we note the following. 
\begin{remark}\label{rem:1+epsilon}
    {\em
Assume that $(f,G)$ satisfies the conditions \ref{condition:convexity}, \ref{condition:subspacecontained}, \ref{condition:existencecones}, \ref{condition:cornersorthogonal} on $E$ and let $L  = \sup_{y\in E} |G(y)|.$ If we were interested in obtaining a convex and Lipschitz extension $F\in C^1(\R^n)$ of $(f,G)$ that \textit{merely} satisfies    
    $$
    \lip(F) \leq \varepsilon + \sup_{y\in E} |G(y)| ,
    $$
    for $\varepsilon>0$, and with $X_F$ not necessarily prescribing a given global behavior $X \subset \R^n$, it would suffice to consider a $C^1(\R^n)$ convex extension $\widetilde{F}$ of $(f,G),$ (e.g. applying directly \cite[Theorem 1.14]{AM19APDE} to $(f,G)$) and then define
    $$
    F(x) = \inf\lbrace \widetilde{F}(y)+ (L + \varepsilon) |x-y| \, : \, y\in \R^n \rbrace, \quad x\in \R^n.
    $$
This can be seen by inspection of the proof Lemma \ref{lem:finalextension_C1}, which would be easier in that case. It is our purpose to construct an extension $F$ with the sharp constant $\lip(F) = L$ that has led us to the work in Sections \ref{sect:behavior} and \ref{sect:proofmain}.  
    }
\end{remark}

\section*{Acknowledgements}

I am supported by the Marie Skłodowska-Curie (MSCA-EF) European Fellowship, grant number 101151594, from the Horizon Europe Funding program. 

\medskip

\bibliographystyle{amsplain}
\bibliography{main}

\end{document}